\documentclass[11pt]{amsart}
\usepackage{amsmath,amsthm}
\usepackage{amssymb,latexsym}
\usepackage{enumerate}

\newtheorem{thm}{Theorem} 
\newtheorem{lem}{Lemma}

\newtheorem{cor}[thm]{Corollary}

\theoremstyle{definition}

\numberwithin{equation}{section}

\makeatletter
\@namedef{subjclassname@2010}{%
  \textup{2010} Mathematics Subject Classification}
\makeatother

\newcommand{\Leg}[2]{\left(\frac{#1}{#2}\right)}

\newcommand{\Z}{\mathbb{Z}}

\newcommand{\Oc}{\mathcal{O}}

\begin{document}


\title[Explicit upper bound for the average number of divisors]{
Explicit upper bound\\ 
for the average number of divisors \\of irreducible quadratic polynomials}

\author[K. Lapkova]{Kostadinka Lapkova}
\address{
Graz University of Technology\\
Institute of Analysis and Number Theory\\
Kopernikusgasse 24/II, 8010 Graz, Austria}
\email{lapkova@math.tugraz.at}

\date{12.09.2017}

\begin{abstract} Consider the divisor sum $\sum_{n\leq N}\tau(n^2+2bn+c)$ for integers $b$ and $c$. We extract an asymptotic formula for the average divisor sum in a convenient form, and provide an explicit upper bound for this sum with the correct main term. As an application we give an improvement of the maximal possible number of $D(-1)$-quadruples.
\end{abstract}


\maketitle


\section{Introduction}\label{sec:intro}

\hspace{0.5cm} Let $\tau(n)$ denote the number of positive divisors of the integer $n$ and $P(x)\in\Z[x]$ be a polynomial. Due to their numerous applications average sums of divisors 
\begin{equation}\label{S1}\sum_{n=1}^N\tau\left(P(n)\right)
\end{equation}
have obtained a lot of attention, e.g. in \cite{delmer,elsh-tao,erdos,hooley1,scour1,scour2} . The current paper aims to improve the author's results from \cite{lapkova} concerning explicit upper bound for divisor sums over certain irreducible quadratic polynomials $P(x)$. More precisely, while the upper bound obtained in \cite{lapkova} appears to be the first one of the right order of magnitude $N\log N$ with completely explicit constants, it does not provide the expected constant in the main term. Here we fill this gap achieving the correct main term as in the asymptotic formula for the corresponding divisor sum. A more detailed introduction to the subject can be found in \cite{lapkova} and \cite{mckee3}.   

\par Only very recently \cite{dudek} and \cite{lapkova2} provided asymptotic formulae for the sum (\ref{S1}) for some \emph{reducible} quadratic polynomials, whereas the asymptotic formulae for \emph{irreducible} quadratic polynomials are classical, due to Scourfield \cite{scour1}, Hooley \cite{hooley} and McKee \cite{mckee}, \cite{mckee2}, \cite{mckee3}. 
\par Let $P(x)=x^2+Bx+C$ for integers $B$ and $C$ such that $P(x)$ is irreducible. Denote 
\begin{equation}\label{def:rho}
\rho(d,P)=\rho(d):=\#\left\{0\leq m<d : P(m)\equiv 0\pmod d \right\}.
\end{equation}
Then we have
$$\sum_{n\leq N}\tau(P(n))\sim \lambda N\log N,$$
as $N\rightarrow\infty$, for some $\lambda$ depending on $B$ and $C$. Hooley \cite{hooley} showed that for irreducible $P(x)=x^2+C$ we have 
$$\lambda=\frac{8}{\pi^2}\sum_{\alpha=0}^\infty \frac{\rho(2^\alpha)}{2^\alpha}\sum_{\substack{d^2\mid C\\d \text{ odd}}}\frac{1}{d}\sum_{\substack{l=1\\l \text{ odd}}}^\infty \frac{1}{l}\Leg{-C/d^2}{l},$$
while for the more general $P(x)=x^2+Bx+C$ with $\Delta=B^2-4C$ not a square, McKee \cite{mckee3} obtained
\begin{equation}\label{eq:mckee}
\lambda=\left\lbrace\begin{array}{ll}
12H^*(\Delta)/(\pi\sqrt{|\Delta|})& \text{if }\Delta <0,\\
12H^*(\Delta)\log{\epsilon_\Delta}/(\pi^2\sqrt{\Delta})& \text{if }\Delta >0,
\end{array}\right.
\end{equation}
where $H^*(\Delta)$ is a weighted class number, and $\epsilon_\Delta$ for $\Delta>0$ is a fundamental unit. Nevertheless we would prove independently the asymptotic formula for the specialized irreducible quadratic polynomials in which we are interested. We achieve a compact main term with a constant $\lambda$ which is easier to formulate and easy to compare with the main term in our explicit upper bound.

\par Our first result is the following theorem. 


\begin{thm}\label{thm1}Let $b$ and $c$ be integers, such that the discriminant $\delta:=b^2-c$ is non-zero and square-free, and $\delta\not\equiv 1\pmod 4$.  Then for $N\rightarrow\infty$ we have the asymptotic formula
$$\sum_{n=1}^N \tau(n^2+2bn+c)=\frac{2}{\zeta(2)}L(1,\chi) N\log N + \Oc(N),$$
where $\chi(n)=\Leg{4\delta}{n}$ for the Kronecker symbol $\Leg{.}{.}$.
\end{thm}
In particular, when we consider the polynomial $f(n)=n^2+1$ we have $\delta=-1$ and $\chi(n)=\Leg{-4}{n}$ is the non-principal primitive Dirichlet character modulo $4$. Then $L(1,\chi)=\sum_{n=0}^\infty (-1)^n/(2n+1)=\pi/4$ by the formula of Leibniz. We substitute $\zeta(2)=\pi^2/6$ in Theorem \ref{thm1} and recover the well-known asymptotic
\begin{equation}
\sum_{n=1}^N \tau(n^2+1)
=\frac{3}{\pi}N\log N +\Oc(N).
\end{equation}

We note that $4\delta=\Delta$ is a fundamental discriminant and the weighted class number $H^*(\Delta)$ in McKee's formula \eqref{eq:mckee} is the usual class number for $\Delta>0$ and $\Delta<-4$. When $\delta=-1$ we have an adjustment by the corresponding root number. Indeed, for negative discriminants $H^*(\Delta)$ is the Hurwitz class number for which the precise relation with the usual class number is described for example in [\cite{cohen}, Lemma $5.3.7$]. Then the Dirichlet class number formula yields the equality of the constant $2L(1,\chi)/\zeta(2)$ with $\lambda$ from \eqref{eq:mckee} for the polynomials we consider in Theorem \ref{thm1}. Our proof, however, is different than McKee's \cite{mckee}, \cite{mckee2}, \cite{mckee3}, and is rather similar to Hooley's argument from \cite{hooley}.

The novelty in this paper is the following explicit upper bound, which achieves the correct main term as in the asymptotic formula from Theorem \ref{thm1}.


\begin{thm}\label{thm2}
Let $f(n)=n^2+2bn+c$ for integers $b$ and $c$, such that the discriminant $\delta:=b^2-c$ is non-zero and square-free, and $\delta\not\equiv 1\pmod 4$.  Assume also that for $n\geq 1$ the function $f(n)$ is non-negative. Then for any $N\geq 1$ satisfying $f(N)\geq f(1)$, and $X:=\sqrt{f(N)}$, we have the inequality
\begin{align*}
\sum_{n=1}^N \tau(n^2+2bn+c)&\leq\frac{2}{\zeta(2)}L(1,\chi) N\log X \\
&+ \left(2.332L(1,\chi)+\frac{4M_\delta}{\zeta(2)}\right)N+\frac{2M_\delta}{\zeta(2)} X\\
&+4\sqrt{3}\left(1-\frac 1{\zeta(2)}\right)M_\delta\frac{N}{\sqrt{X}}\\
&+2\sqrt{3}\left(1-\frac 1{\zeta(2)}\right)M_\delta\sqrt{X}
\end{align*}
where $\chi(n)=\Leg{4\delta}{n}$ for the Kronecker symbol $\Leg{.}{.}$ and 
\begin{equation*}
M_\delta=\left\lbrace
\begin{array}{ll}
\frac{4}{\pi^2}\delta^{1/2}\log{4\delta}+1.8934\delta^{1/2}+1.668 &,\text{ if } \delta>0;\\
 & \\
\frac{1}{\pi}|\delta|^{1/2}\log{4|\delta|}+1.6408|\delta|^{1/2}+ 1.0285&,\text{ if } \delta<0.
\end{array}\right.
\end{equation*}
\end{thm}

Note that we can formulate the theorem by extracting positive constants $C_1, C_2, C_3$ such that 
\begin{equation*}
\sum_{n=1}^N \tau(n^2+2bn+c)\leq \frac{2}{\zeta(2)}L(1,\chi) N\log N+C_1N+C_2\sqrt{N}+C_3,
\end{equation*}
because $X=\sqrt{f(N)}=N+\Oc(1)$. However we restrained ourselves from doing so in order to keep our result more precise for eventual numerical applications. Still, improvements of the lower order terms are very likely, for example by applying the theory developed by Akhilesh and Ramar\'e \cite{ramareAkhilesh}.

\par When we know the exact form of the quadratic polynomial and the corresponding character, we might achieve even better upper bounds. This is the case of the polynomial $f(n)=n^2+1$ when the following corollary holds.


\begin{cor}\label{cor1}
For any integer $N\geq 1$ we have
$$\sum_{n\leq N}\tau(n^2+1)<\frac{3}{\pi}N\log N +3.0475 N+1.3586 \sqrt{N}.$$
\end{cor}

Just as in \cite{lapkova} we give an application of the latter inequality. Define a \textit{$D(n)-m$-tuple} for a nonzero integer $n$ and a positive integer $m$ to be a set of $m$ integers such that the product of any two of them increased by $n$ is a perfect square. It is conjectured that there are no $D(-1)$-quadruples but currently it is only known that there are at most $4.7\cdot 10^{58}$ $D(-1)$-quadruples \cite{lapkova}. The latter number upgraded the previous maximal possible bound $3.01\times 10^{60}$ due to Trudgian \cite{trudgian}, which in turn improved results from \cite{cipu} and \cite{elsh}.  

\par Plugging the upper bound of Corollary \ref{cor1} in the proof of [\cite{elsh}, Theorem 1.3] from the paper of Elsholtz, Filipin and Fujita we obtain another slight improvement.


\begin{cor}\label{thm:Dm1}
There are at most $3.677\cdot 10^{58}$ $D(-1)$-quadruples.
\end{cor}


\section{Proof of Theorem \ref{thm2}}\label{sec:2}

We start with proving the explicit upper bound in Theorem \ref{thm2}, as the claims required for the proof of Theorem \ref{thm1} can be easily adapted by the lemmae from this section.

\par Let us consider the polynomial $f(n)=n^2+2bn+c$ with integer coefficients $b$ and $c$, and the function $\rho(d)$ defined in \eqref{def:rho} counts the roots of $f(n)$ in a full residue system modulo $d$. Let $\delta=b^2-c$ and $\chi(n)=\Leg{4\delta}{n}$ for the Kronecker symbol $\Leg{.}{.}$.

\par The core of the proofs of both Theorem \ref{thm1} and Theorem \ref{thm2} is the following convolution lemma which we proved in \cite{lapkova}.

 
\begin{lem}[\cite{lapkova}, Lemma 2.1]\label{lem1} Let $\delta=b^2-c$ be square-free and $\delta\not\equiv 1\pmod 4$. Then we have the identity
$$ \rho(d)=\sum_{lm=d}\mu^2(l)\chi(m).$$
\end{lem}
The proof of Lemma \ref{lem1} is based on elementary facts about the function $\rho(d)$ at prime powers and manipulations of the Dirichlet series it generates.

\par Further we need the following explicit estimates.
 
\begin{lem}\label{lem:muQexpl} For any integer $N\geq 1$ we have 
$$\sum_{n\leq N}\mu^2(n)=\frac{N}{\zeta(2)}+E_1(N),$$
where $|E_1(N)|\leq \sqrt{3}\left(1-\frac 1{\zeta(2)}\right)\sqrt{N}<0.6793\sqrt{N}$.
\end{lem}

\begin{proof} This is inequality (10) from Moser and MacLeod \cite{moser}.
%
 \end{proof}
The following numerically explicit P\'{o}lya-Vinogradov inequality is essentially proven by Frolenkov and Soundararajan \cite{frolenkov}. It supercedes the main result of Pomerance \cite{pomerance}.
\begin{lem}\label{lem:pomerance} Let $$M_\chi:=\max_{L,P}\left|\sum_{n=L}^P \chi(n)\right|$$ for a primitive character $\chi$ to the modulus $q>1$. Then
\begin{equation*}
M_\chi\leq\left\lbrace
\begin{array}{ll}
\frac{2}{\pi^2}q^{1/2}\log{q}+0.9467q^{1/2} +1.668\,, & \chi \text{ even;}\\
 & \\
\frac{1}{2\pi}q^{1/2}\log{q}+0.8204q^{1/2}+1.0285,&\chi \text{  odd. }
\end{array}\right.
\end{equation*} 
\end{lem}
\begin{proof}
Both inequlities for $M_\chi$ are shown to hold by Frolenkov and Soundararajan in the course of the proof of their Theorem 2 \cite{frolenkov} as long as a certain parameter $L$ satisfies $1\leq L\leq q$ and $L=\left[\pi^2/4\sqrt{q}+9.15\right]$ for $\chi$ even, $L=\left[\pi\sqrt{q}+9.15\right]$ for $\chi$ odd. Thus both inequalities for $M_\chi$ hold when $q>25$. \\
Then we have a look of the maximal possible values of $M_\chi$ when $q\leq 25$ from a data sheet, which was kindly provided by Leo Goldmakher. It represents the same computations of Bober and Goldmakher used by Pomerance \cite{pomerance}. We see that the right-hand side of the bounds of Frolenkov-Soundararajan for any $q\leq 25$ is larger than the maximal value of $M_\chi$ for any primitive Dirichlet character $\chi$ of modulus $q$. This proves the lemma.   
\end{proof}
The next lemma is critical for obtaining the right main term in the explicit upper bound of Theorem \ref{thm2}.
\begin{lem}\label{lem:chiexpl} For any integer $N\geq 1$ we have 
$$\displaystyle\sum_{n\leq N}\frac{\chi(n)}{n}=L(1,\chi)+E_2(N),$$
where $\left|E_2(N)\right|\leq 2M_\delta/N$ and $M_\delta$ is the constant defined in Theorem \ref{thm2}.
\end{lem}
\begin{proof} This statement, especially the estimate of the error term, is much less trivial and follows from the effective P\'{o}lya-Vinogradov inequality of Frolenkov-Soundararajan. 
First we notice that 
$$\sum_{n\leq N}\frac{\chi(n)}{n}=\sum_{n=1}^\infty\frac{\chi(n)}{n}-\sum_{n>N}\frac{\chi(n)}{n}=L(1,\chi)+E_2(N).$$
Let us denote 
$X(N):=\sum_{1\leq n\leq N}\chi(n)$ for any positive integer $N$. Then by Abel's summation it follows that for any positive integer $Z>N$ we have
\begin{align*}
\left|\sum_{N<n\leq Z}\frac{\chi(n)}{n}\right|&=\left|\frac{X(Z)}{Z}-\frac{X(N)}{N}+\int_N^Z\frac{X(t)}{t^2}dt\right|\\
&\leq \frac{\left|X(Z)\right|}{Z}+\frac{\left|X(N)\right|}{N}+\int_N^Z\frac{\left|X(t)\right|}{t^2}dt .
\end{align*}  
Recall the definition of the quantity $M_\chi$ in Lemma \ref{lem:pomerance}. Then 
$$\left|\sum_{N<n\leq Z}\frac{\chi(n)}{n}\right|\leq M_\chi\left(\frac 1 Z+\frac 1 N+\int_N^Z\frac{dt}{t^2}\right)=\frac{2M_\chi}{N},$$
which is uniform in $Z$. Therefore $|E_2(N)|=\left|\sum_{n>N}\chi(n)/n\right|\leq 2M_\chi/N$. Now by the conditions on $\delta$ to be square-free and not congruent to $1$ modulo $4$ it follows that the discriminant $4\delta$ is fundamental, thus the character $\chi(n)=\Leg{4\delta}{n}$ is primitive with a conductor $4|\delta|$. Then according to Lemma \ref{lem:pomerance} $M_\chi\leq M_\delta$ and this proves the statement. \end{proof}

The following lemma is due to Ramar\'e. One could easily extract the right main term from Lemma \ref{lem:muQexpl} through Abel summation but the estimate of the minor term requires computer calculations.
 
\begin{lem}[\cite{ramare}, Lemma 3.4]\label{lem:ramare} Let $x\geq 1$ be a real number.We have 
$$ \sum_{n\leq x}\frac{\mu^2(n)}{n}\leq \frac{\log x}{\zeta(2)}+1.166.$$
\end{lem}

 
Let us start with the proof of Theorem \ref{thm2}. Note that from the condition $f(N)\geq f(1)$ and the convexity of the function $f(x)$ we have $f(N)=\max_{1\leq n\leq N}f(n)$. Using the Dirichlet hyperbola method we have
\begin{align*}
\sum_{1\leq n\leq N}\tau(f(n))&=\sum_{1\leq n\leq N}\sum_{d\mid f(n)}1\leq 2\sum_{1\leq n\leq N}\sum_{\substack{d\leq\sqrt{f(n)}\\d\mid f(n)}}1\\
&=2\sum_{1\leq d\leq \sqrt{f(N)}}\sum_{\substack{1\leq n\leq N\\ d+\Oc(1)\leq n}}\sum_{d\mid f(n)}1\leq 2\sum_{1\leq d\leq \sqrt{f(N)}}\sum_{\substack{1\leq n\leq N\\d\mid f(n)}}1.
\end{align*} 
Recall the definition \eqref{def:rho} of the function $\rho(d)$. Then the innermost sum equals $\left[N/d\right]$ copies of $\rho(d)$ plus a remaining part smaller than $\rho(d)$. Recall that $X=\sqrt{f(N)}$. Then 
\begin{equation}\label{eq:tau_rho}
\sum_{1\leq n\leq N}\tau(f(n))\leq 2\sum_{1\leq d\leq X}\left(\frac N d \rho(d)+\rho(d)\right)=2N\sum_{1\leq d\leq X}\frac{\rho(d)}d +2\sum_{1\leq d\leq X}\rho(d).
\end{equation}
From the convolution identity stated in Lemma \ref{lem1} it immediately follows that
$$\sum_{d\leq X}\rho(d)=\sum_{lm\leq X}\mu^2(l)\chi(m)=\sum_{l\leq X}\mu^2(l)\sum_{m\leq X/l}\chi(m).$$
Recall the definition of $M_\chi$ and the Frolenkov-Soundararajan bounds $M_\chi\leq M_\delta$. Then using also Lemma \ref{lem:muQexpl} we get
\begin{equation}\label{eq:rho_sum}
\sum_{d\leq X}\rho(d)\leq M_\delta\sum_{l\leq X}\mu^2(l)\leq M_\delta\left(\frac{X}{\zeta(2)}+\sqrt{3}\left(1-\frac 1{\zeta(2)}\right)\sqrt{X}\right).
\end{equation}
Similarly by the convolution property of the function $\rho(d)$ and Lemma \ref{lem:chiexpl} we can write
\begin{equation}\label{eq:rho_d_sum}
\sum_{d\leq X}\frac{\rho(d)}{d}=\sum_{l\leq X}\frac{\mu^2(l)}{l}\sum_{m\leq X/l}\frac{\chi(m)}{m}\leq \sum_{l\leq X}\frac{\mu^2(l)}{l}\left(L(1,\chi)+2M_\delta\frac{l}{X}\right).
\end{equation}
One could go further using Lemma \ref{lem:ramare} and again Lemma \ref{lem:muQexpl}. We obtain
\begin{align}\label{eq:rhod_sum}
\sum_{d\leq X}\frac{\rho(d)}{d}&\leq L(1,\chi)\sum_{l\leq X}\frac{\mu^2(l)}{l}+\frac{2M_\delta}{X}\sum_{l\leq X}\mu^2(l)\nonumber\\
&\leq L(1,\chi)\left(\frac{\log X}{\zeta(2)}+1.166\right)+\frac{2M_\delta}{X}\left(\frac{X}{\zeta(2)}+\sqrt{3}\left(1-\frac 1{\zeta(2)}\right)\sqrt{X}\right).
\end{align}
Plugging the estimates \eqref{eq:rho_sum} and \eqref{eq:rhod_sum} in \eqref{eq:tau_rho} gives the statement of Theorem \ref{thm2}.

\par Let us prove Corollary \ref{cor1}. For the polynomial $f(n)=n^2+1$ we can be more precise as the condition $d\leq \sqrt{f(n)}$ is equivalent to $d\leq n$. Then after applying more carefully the Dirichlet hyperbola method and changing the order of summation we obtain
\begin{align*}
\sum_{n\leq N}\tau(n^2+1)&=2\sum_{n\leq N}
\sum_{\substack{d\leq n\\d\mid n^2+1}}1=2\sum_{1\leq d\leq N}\sum_{\substack{d\leq n\leq N\\d\mid n^2+1}}1\\
&=2\sum_{1\leq d\leq N}\left(\sum_{\substack{1\leq n\leq N\\d\mid n^2+1}}1-\sum_{\substack{1\leq n< d\\d\mid n^2+1}}1\right)\leq 2N\sum_{1\leq d\leq N}\frac{\rho(d)}{d}.
\end{align*}
Similarly to \eqref{eq:rho_d_sum} we have
$$\sum_{1\leq d\leq N}\frac{\rho(d)}{d}=\sum_{l\leq N}\frac{\mu^2(l)}{l}\sum_{m\leq N/l}\frac{\chi(m)}{m}\leq \sum_{l\leq N}\frac{\mu^2(l)}{l}\left(L(1,\chi)+\frac{l}{N}\right)$$
becuase we can see that in this case $\left|\sum_{n>N}\chi(n)\right|\leq 1/N$. Furthermore $L(1,\chi)=\pi/4$ and application of Lemma \ref{lem:ramare} and Lemma \ref{lem:muQexpl} gives 
\begin{align*}
\sum_{1\leq d\leq N}\frac{\rho(d)}{d}&\leq \frac{\pi}{4}\sum_{l\leq N}\frac{\mu^2(l)}{l}+\frac{1}{N}\sum_{l\leq N}\mu^2(l)\\
&\leq \frac{\pi}{4}\left(\frac{\log N}{\zeta(2)}+1.166\right)+\frac{1}{N}\left(\frac{N}{\zeta(2)}+\sqrt 3\left(1-\frac{1}{\zeta(2)}\right)\sqrt N\right).\end{align*}
After numerical approximation of the constants we conclude the validity of Corollary \ref{cor1}.

\section{Proof of Theorem \ref{thm1}}

Let again $f(n)=n^2+2bn+c$. We can assume that $f(n)$ is non-negative for $n\geq 1$ as for $N$ large enough $f(N)=\max_{1\leq n\leq N}f(n)$ and the contribution of the finitely many integers $n$ for which $f(n)<0$ will be negligible compared to the main term.  
\par As in the proof of Theorem \ref{thm2} and using \eqref{eq:tau_rho} and \eqref{eq:rho_sum} it is easy to see that
\begin{align}\label{eq:Trho}
\sum_{n\leq N}\tau(f(n))&=2N\sum_{d\leq N}\frac{\rho(d)}{d}+\Oc\left(\sum_{d\leq N}\rho(d)\right)+\Oc(N)\nonumber\\&=2N\sum_{d\leq N}\frac{\rho(d)}{d}+\Oc(N).
\end{align} 

We need the following estimates.

\begin{lem}\label{lem:asympt}For $N\rightarrow\infty$ we have the asymptotic formulae
$$\sum_{n\leq N}\frac{\mu^2(n)}{n}=\frac{\log N}{\zeta(2)}+\Oc(1)$$
and
 $$\displaystyle\sum_{n\leq N}\frac{\chi(n)}{n}=L(1,\chi)+\Oc(1/N).$$
\end{lem}
\begin{proof}
The first asymptotic formula follows easily after Abel transformation of the classical formula
$$\sum_{n\leq N}\mu^2(n)=\frac{N}{\zeta(2)}+\Oc(\sqrt N),$$
which in its turn obviously follows from the explicit version in Lemma \ref{lem:muQexpl}. 
The second statement can be deduced from Lemma \ref{lem:chiexpl}. 
%
\end{proof}

\par Then the asymptotic estimate corresponding to \eqref{eq:rho_d_sum} is
\begin{align}\label{eq:sumrhofrac}\sum_{d\leq N}\frac{\rho(d)}{d}&=\sum_{l\leq N}\frac{\mu^2(l)}{l}\left(L(1,\chi)+\Oc(l/N)\right)=L(1,\chi)\sum_{l\leq N}\frac{\mu^2(l)}{l}\nonumber\\
&+\Oc\left(\frac 1 N\sum_{l\leq N}\mu^2(l)\right)=L(1,\chi)\left(\frac{\log N}{\zeta(2)}+\Oc(1)\right)+\Oc(1)\nonumber\\
&=\frac{L(1,\chi)}{\zeta(2)}\log N+\Oc(1).
\end{align}

Now plugging \eqref{eq:sumrhofrac} into \eqref{eq:Trho} gives Theorem \ref{thm1}.

\section{Some examples}
In our preceding paper \cite{lapkova} we gave several examples of polynomials for which Theorem \ref{thm2} can provide effective upper bound of the average divisor sum. There we compared the coefficient of the main term in the asymptotic formula of McKee with the one from our upper bound. The current paper achieves equality of the two coefficients but there is one more computation to be made, so that our result could be useful numerically, this of the special value of the Dirichlet $L$-function $L(s,\chi)$ at $s=1$. 
 
\par Formulae for computing $L(1,\chi)$ in the two separate cases $\delta<0$ and $\delta>0$ are given in [\cite{washington}, Theorem 4.9]. This is implemented by the \texttt{SAGE} function \texttt{quadratic\_L\_function\_exact$(1,4\delta)$} when $\delta<0$ and its values are computed for several examples in the table below.
 When $\delta>0$ one can find the approximate value of $L(1,\chi)$ by finding first the corresponding Dirichlet character (of modulus $4\delta$, order $2$, even and primitive) in the \textit{Refine search} platform from LMFDB (The $L$-functions and modular forms database \cite{LMFDB}).
 \par The following table lists some examples.

 \[ 
\def\arraystretch{1.7}
\arraycolsep=2.5pt
\begin{array}{| c | c | c | c | }
    \hline
    f(n) & \delta & \chi & L(1,\chi) \\ \hline
    n^2+1 &-1 &\Leg{-1}{.} & \pi/4\sim 0.7854 \\ \hline
    n^2+10n+27 & -2& \Leg{-8}{.}  & \sqrt 2/4\pi\sim 1.1108\\    \hline
    n^2+4n+10 & -6 & \Leg{-24}{.} & \sqrt 6/6\pi\sim 1.2826 \\    \hline
    n^2+10n-10 & 35 & \Leg{140}{.} & \sim 0.8377 \\ \hline
 n^2+20n+9 & 91 & \Leg{364}{.} & \sim 1.6887\\ \hline
  \end{array}
\]
\vspace{0.2cm}
\subsection*{Funding.} 
This work was supported by a Hertha Firnberg grant of the  Austrian Science Fund (FWF) [T846-N35].

\subsection*{Acknowledgments.} 
The author thanks Christian Elsholtz for drawing her attention to this problem and Olivier Bordell\`es, Dmitry Frolenkov and Olivier Ramar\'e for their useful comments on previous versions of this note. The author is also very grateful to Leo Goldmakher for kindly providing the data used in the proof of Lemma \ref{lem:pomerance}.

\end{document}